\documentclass[]{article}
\usepackage{amsmath}
\usepackage{amsfonts,amsmath,amsthm,amssymb}
\usepackage{mathtools}
\usepackage[numbers,sort]{natbib}
\usepackage{geometry}
\usepackage{pgf,pgfplots}
\usepackage{tikz}
\usepackage{caption}
\usepackage{subcaption}
\usepackage[]{enumitem}
\usepackage{bigints}
\usepackage{float}
\usepackage[title]{appendix}

\usepackage{enumitem}

\usepackage{bm}
\usepackage{algorithm,algpseudocode}

\usepackage{xcolor}
\RequirePackage[colorlinks,citecolor=blue,urlcolor=blue]{hyperref}

\theoremstyle{plain}
\newtheorem{lem}{Lemma}
\newtheorem{cor}[lem]{Corollary}
\newtheorem{prop}[lem]{Proposition}
\newtheorem{thm}[lem]{Theorem}

\theoremstyle{definition}

\newtheorem{defn}[lem]{Definition}

\newtheorem{assum}[lem]{Assumption}

\newcommand{\R}{\mathbb{R}}
\renewcommand{\P}{\mathbb{P}}

\newcommand{\E}{\mathbb{E}}

\newcommand{\N}{\mathbb{N}}

\renewcommand{\( }{\left(}

\newcommand{\Bcal}{{\mathcal B}}

\makeatletter
\renewenvironment{proof}[1][\proofname] {\par\pushQED{\qed}\normalfont\topsep6\p@\@plus6\p@\relax\trivlist\item[\hskip\labelsep\bfseries#1\@addpunct{.}]\ignorespaces}{\popQED\endtrivlist\@endpefalse}
\makeatother

\begin{document}
	\title{Generalized Rank Dirichlet Distributions}
	\author{David Itkin\footnote{Department of Mathematics, Imperial College London, \url{d.itkin@imperial.ac.uk}}}

	\maketitle

	\begin{abstract}
		We study a new parametric family of distributions on the ordered simplex $\nabla^{d-1} = \{y \in \R^d: y_1 \geq \dots \geq y_d \geq 0, \ \sum_{k=1}^d y_k = 1\}$, which we call \emph{Generalized Rank Dirichlet (GRD)} distributions. Their density is proportional to $\prod_{k=1}^d y_k^{a_k-1}$ for a parameter $a = (a_1,\dots,a_d) \in \R^d$ satisfying $a_k + a_{k+1} + \dots + a_d > 0$ for $k=2,\dots,d$. The density is similar to the Dirichlet distribution, but is defined on $\nabla^{d-1}$, leading to different properties. In particular, certain components $a_k$ can be negative.
		Random variables $Y = (Y_1,\dots,Y_d)$ with GRD distributions have previously been used to model capital distribution in financial markets and more generally can be used to model ranked order statistics of weight vectors.
		 We obtain for any dimension $d$ \emph{explicit expressions} for moments of order $M \in \N$ for the $Y_k$'s and moments of all orders for the log gaps $Z_k = \log Y_{k-1} - \log Y_k$ when $a_1 + \dots + a_d = -M$. Additionally, we propose an algorithm to \emph{exactly} simulate  random variates in this case. In the general case $a_1 + \dots + a_d \in \R$ we obtain series representations for these quantities and provide an approximate simulation algorithm.
	\end{abstract}

\section{Introduction}
For an integer $d \geq 2$ we study a parametric family of distributions defined on the ordered simplex 
\[\nabla^{d-1} = \{y \in \R^d: y_1 \geq y_2\geq \dots \geq y_d \geq 0\quad \text{ and } \quad y_1+\dots+y_d = 1\},\]
 whose density is proportional to \begin{equation} \label{eqn:density}
 	\prod_{k=1}^d y_k^{a_k-1}
 \end{equation}  for a parameter $a \in \R^d$. It was shown in \cite{itkin2021open} (and reproduced below in Proposition~\ref{prop:finite}) that this density induces a probability measure on $\nabla^{d-1}$, when appropriately normalized, if and only if 
 \begin{equation} \label{eqn:bar_a}
 	\bar a_k := a_k + a_{k+1} + \dots + a_d > 0, \quad \text{for } k=2,\dots,d.
 \end{equation} 
Notably, condition \eqref{eqn:bar_a} allows for certain $a_k$'s to be negative as long as the tail sum $\bar a_k$ remains positive. In fact, even parameters satisfying $a_k < 0$ for $k=1,\dots,d-1$ can be compatible with condition \eqref{eqn:bar_a} (as long as $a_d$ is sufficiently positive). 

In the special case $a_1 = a_2 = \dots = a_d > 0$, if $X \sim  \mathrm{Dirichlet}(a)$ then the ranked vector of decreasing order statistics $Y =(X_{(1)},\dots,X_{(d)})$ has density proportional to \eqref{eqn:density}. In the case that the components of $a$ are not all the same this relationship is no longer true. However, since the functional form of \eqref{eqn:density} is the same as for the Dirichlet density -- just defined on the ordered simplex rather than the standard simplex -- we call the induced probability distribution the \emph{generalized ranked Dirichlet} distribution with parameter $a$, or GRD($a$) for short.

 The GRD distribution can be used to model the distribution of ranked weight vectors that sum to one even for a general $a$ parameter. Indeed, if $X=(X_1,\dots,X_d)$ is a random (unordered) vector of nonnegative weights that sum to one with density proportional to $\prod_{k=1}^d x_{(k)}^{a_k-1}$ then the decreasing order statistics $Y=(X_{(1)},\dots,X_{(d)})$ follow a GRD($a$) distribution.

 To the best of the author's knowledge the general form of the GRD($a$) distribution under the condition \eqref{eqn:bar_a} first appeared as the invariant density of a certain stochastic process, called a \emph{rank Jacobi} process in \cite{itkin2021open}. Previously, the special case with $\bar a_1 = \sum_{k=1}^d a_k = 0$ had appeared in \cite{banner2005atlas,pal2008one,ichiba2011hybrid,fernholz2002stochastic}, where it arose as the invariant measure to a class of processes known as \emph{Atlas} or \emph{first-order models}. In particular, in \cite{banner2005atlas}, a connection to independent exponential random variables via the log gaps (see equation \eqref{eqn:log_gaps} below) was established. The analysis in this paper heavily exploits this relationship to exponential random variables in the case $\bar a_1 = 0$ to study GRD($a$) distributions for more general parameters $a$. 
 
 Arguably, the most well-studied distribution that models ranked weights is the \emph{Poisson--Dirichlet} \emph{(PD)} distribution introduced by Kingman in \cite{kingman1975random}. Indeed, it has found applications in a large number of fields including population genetics, number theory, physics, finance and statistics (see \cite{pitman1997two,feng2010poisson} for detailed accounts of the PD distribution). However, it is defined on the infinite dimensional \emph{Kingman simplex} $\{y \in \R^\infty: y_1 \geq y_2 \geq \dots \geq 0, \ \sum_{k=1}^\infty y_k = 1\}$ and as such is an infinite-dimensional distribution. In the author's PhD thesis \cite{itkin2022growth}, it was shown that, under appropriate assumptions on the parameter vector, the GRD distribution converges as $d \to \infty$ to a distribution on the Kingman simplex which is absolutely continuous with respect to a PD distribution with an explicitly given density. As such, the GRD family can be viewed as a finite dimensional relative of the PD distribution.

 Remarkably, even in the most basic case $d = 2$, the GRD distribution does not in general seem to be a standard probability distribution with a previously recorded name. When $d=2$ we can write $Y_2 = 1-Y_1$ and reduce to a one-dimensional random variable $Y_1$, which has density proportional to
\[y^{a_1-1}(1-y)^{a_2-1}, \quad y \in [1/2,1].\]
When $a_1 > 0$ this coincides with a truncated Beta distribution, but the case $a_1 \leq 0$ does not seem to have an established name.


 Nevertheless, this distribution has remarkable structural properties. In Section~\ref{sec:definition} we formally define the GRD distribution. Under the condition $\bar a_1 = 0$ the aforementioned relationship to independent exponential distributions is explored in Section~\ref{sec:bar_a1=0}, which we use to obtain \emph{negative} moments of all orders for the largest weight $Y_1$. In Section~\ref{sec:com} we then obtain a change of measure identity which establishes a relationship between GRD distributions with different parameters.
In Section~\ref{sec:bar_a1=-M} we explore the case $\bar a_1 = -M$ for some positive integer $M$. In this case the change of measure formula can be leveraged to obtain explicit expressions for the \emph{positive} moments of the $Y_k$'s up to order $M$, which are derived in Section~\ref{sec:positive_moments}. In particular, when $M = 1$, the moment formula is invertible with respect to the parameter vector $a$ allowing for explicit first moment matching. Additionally, it is shown in Section~\ref{sec:log_gaps} that the \emph{log gaps} 
\begin{equation} \label{eqn:log_gaps}
	Z_k = \log Y_{k-1} - \log Y_k, \quad \text{for } k=2,\dots,d 
\end{equation} can be represented as a mixture of exponential random variables when $\bar a_1 = -M$. This leads us to explicit formulas for the moment generating function and moments of all orders for the log gaps. Using the log gaps as an intermediary, in Section~\ref{sec:simulate}, we derive an algorithm to simulate exactly from the GRD($a$) distribution in the case $\bar a_1 = -M$. The general case when $\bar a_1$ is not assumed to be a negative integer is studied in Section~\ref{sec:general}. In this case we obtain a series representation for moments of the log gaps and leverage this to propose an approximate simulation algorithm to generate GRD($a$) random variates.
\paragraph{Notation.} The \emph{tail sum} notation of $\bar a_k = a_k + a_{k+1} + \dots + a_d$, as in \eqref{eqn:bar_a}, is in force throughout the paper. We write $e_1,\dots,e_d$ for the standard basis vectors in $\R^d$. We denote by $\N$ the natural numbers (starting from one) and $\N_0 = \N \cup \{0\}$. For an integer $M > 0$ we define $\N_0^d(M) = \{m \in \N^d_0: \bar m_1 = M\}$. By convention, empty sums are taken to be zero, while empty products are taken to be one. Since $\nabla^{d-1}$ is a $(d-1)$-dimensional subset of $\R^d$, all integrals over $\nabla^{d-1}$ should be understood as the pushforward of Lebesgue measure on $\R^{d-1}$ under the map $(y_1,\dots,y_{d-1}) \mapsto (y_1,\dots,y_{d-1},1-y_1- \dots -y_{d-1})
$.

\section{The GRD Distribution} \label{sec:definition}
Given $a \in \R^d$ we set $Q_a = \int_{\nabla^{d-1}}\prod_{k=1}^d y_k^{a_k-1}\, dy$. Then we have the following result already established in \cite{itkin2021open}. The proof is short and insightful so we reproduce it here.
\begin{prop}[Finite normalizing constant] \label{prop:finite} $Q_a < \infty$ if and only if $\bar a_k > 0$ for $k=2,\dots,d$.
\end{prop}
\begin{proof}
	First note that the size or sign of $a_1$ does not effect integrability of $Q_a$ since $1/d \leq y_1 \leq 1$. Hence we assume without loss of generality that $a_1 = -\bar a_2$. Then we rewrite the integral as 
	\[Q_a =  \int_{\nabla^{d-1}} \prod_{k=2}^{d} \left(\frac{y_{k-1}}{y_{k}}\right)^{-\bar a_k}\prod_{k=1}^d y_k^{-1}\, dy.\]
	 Next consider the change of variables $z_k = \log (y_{k-1}) - \log(y_k)$ for $k=2,\dots,d$. This transformation maps the ordered simplex onto $\R_+^{d-1}$ and its Jacobian is determined by $dz = \prod_{k=1}^{d} y_k^{-1}dy$. Thus we obtain
	\[Q_a = \int_{\R^{d-1}_{+}} \exp\(-\sum_{k=2}^d\bar a_kz_k\)dz = \prod_{k=2}^d \int_0^\infty e^{-\bar a_kz}dz.\]
	This expression is finite if and only if $\bar a_k > 0$ for every $k =2,\dots,d$ completing the proof. 
\end{proof}
This leads us to the standing assumption mentioned in the introduction.
\begin{assum} \label{assum:bar_a}
The parameter vector $a \in \R^d$ satisfies $\bar a_k > 0$ for $k=2,\dots,d$.
\end{assum}
We can now formally define the GRD distribution.
\begin{defn}[Generalized Rank Dirichlet  (GRD) Distribution]
	For a parameter $a \in \R^d$ satisfying Assumption~\ref{assum:bar_a} the probability measure
	\[\P_a(A) = Q_a^{-1}\int_{\nabla^{d-1}} \prod_{k=1}^d y_k^{a_k-1}1_{A}(y)\, dy, \quad A \in \Bcal(\nabla^{d-1})\]
	is called a \emph{Generalized Rank Dirichlet (GRD)} distribution with paremeter $a$. We will write $Y \sim \mathrm{GRD}(a)$ for a random variable $Y$ with law $\P_a$ and denote by $\E_a[\cdot]$ expectation under $\P_a$.
\end{defn} 
\section{The case $\bar a_1 =0$} \label{sec:bar_a1=0}
An important special case of interest is when $\bar a_1  = 0$. In this case a similar calculation as in the proof of Proposition~\ref{prop:finite} shows that the log gaps $(Z_2,\dots,Z_d)$ given by \eqref{eqn:log_gaps} are distributed as independent exponentially distributed random variables whenever $Y \sim \mathrm{GRD}(a)$, and consequently, the weight ratios $Y_{k-1}/Y_k$ follow a Pareto distribution. Moreover, the normalizing constant $Q_a$ is explicitly computable in this case.
To the best of the author's knowledge the Pareto property was first observed in \cite{fernholz2002stochastic} and the relationship to independent exponential random variables was explored in \cite{banner2005atlas}. We collect these results in the following proposition.
\begin{prop}[Section~4 in \cite{banner2005atlas}] \label{cor:bar_a1=0}
	When $\bar a_1 = 0$ we have that $Q_a = \prod_{k=2}^d \bar a_k^{-1}$. Additionally the log gaps $(Z_2,\dots,Z_d)$ are independent and satisfy $Z_k \sim \mathrm{Exp}(\bar a_k)$, while the ratios $Y_{k-1}/Y_k$ are independent and satisfy $Y_{k-1}/Y_k \sim \mathrm{Pareto}(1,\bar a_k)$ for $k=2,\dots,d$.
\end{prop}

These facts can be leveraged to compute certain expected ratios and \emph{negative} moments of $Y_1$.
\begin{thm} \label{thm:bar_a1=0_moments} Let $a \in \R^d$ satisfying Assumption~\ref{assum:bar_a} be given and suppose that $\bar a_1 = 0$.
	\begin{enumerate}
		\item \label{item:ratio_mn} (Moments of ratios) Let $n \in \N_0^d$ and $M \in \N$ such that $M \geq \bar n_1$ be given. Then 
			\begin{equation} \label{eqn:ratio_mn}
			\E_a\left[ \frac{\prod_{k=1}^d Y_k^{n_k}}{Y_1^M} \right] = \sum_{m \in \N_0^d(M-\bar n_1)} {M - \bar n_1\choose m_1,\dots, m_d } \prod_{k=2}^d \frac{\bar a_k}{\bar a_k + \bar m_k + \bar n_k}.
		\end{equation}
		\item \label{item:1/y_1_m} (Negative moments of $Y_1$) For any $M \in \N$, 
		\begin{equation} \label{eqn:1/y_1_m}
			\E_a\left[\frac{1}{Y_1^M}\right] = \sum_{m \in \N_0^d(M)}{M \choose m_1,\dots, m_d }\prod_{k=2}^d \frac{ \bar a_k}{\bar a_k + \bar m_k}.
		\end{equation}
	\end{enumerate}
\end{thm}

\begin{proof}
	First we assume that $\bar n_1 = M$. In this case note that the expectation on the left hand side of \eqref{eqn:ratio_mn} is given by $Q_{a + n - Me_1}/Q_a$. Since $(\overline{a + n -  Me_1})_1  = 0$ we obtain
	\begin{equation} \label{eqn:ratio_m}
		\E_a\left[ \frac{\prod_{k=1}^d Y_k^{n_k}}{Y_1^M} \right] = \E_a\left[\frac{\prod_{k=1}^d Y_k^{n_k}}{Y_1^{\bar n_1}}\right] =  \prod_{k=2}^d \frac{ \bar a_k}{\bar a_k + \bar n_k}
	\end{equation} by Proposition~\ref{cor:bar_a1=0}, which proves \ref{item:ratio_mn} in this case.
	
	To prove \ref{item:ratio_mn} in the general case we use the multinomial formula to obtain
	
	\begin{align*} \E_a\left[ \frac{\prod_{k=1}^d Y_k^{n_k}}{Y_1^M} \right] & = \E_a\left[ \frac{\prod_{k=1}^d Y_k^{n_k}(Y_1+\dots+Y_d)^{M-\bar n_1}}{Y_1^M} \right] \\
		& = \sum_{m \in \N_0^d(M)}  {M - \bar n_1\choose m_1,\dots, m_d } \E_a\left[ \frac{ \prod_{k=1}^d Y_k^{n_k + m_k}}{Y_1^M} \right] \\
		& 
		= \sum_{m \in \N_0^d(M)}  {M - \bar n_1\choose m_1,\dots, m_d } \prod_{k=2}^d \frac{\bar a_k}{\bar a_k + \bar m_k + \bar n_k}.
	\end{align*} 
	In the last equality we used \eqref{eqn:ratio_m}, which is applicable since $\bar n_1 + \bar m_1 = M$. Finally \eqref{eqn:1/y_1_m} follows by taking $n = 0$ in \eqref{eqn:ratio_mn}.
\end{proof}
\section{A change of measure formula} \label{sec:com}
We now derive a change of measure identity, which holds for \emph{any} GRD distribution. This identity is the workhorse for the computations to come.
\begin{thm}[Change of measure] \label{thm:com}
	Fix $a, b \in \R^d$ satisfying Assumption~\ref{assum:bar_a}. Let $f:\nabla^{d-1} \to \R$ be a function that is integrable under $\P_a$. Then 
	\begin{equation} \label{eqn:com}
		\E_a[f(Y)] = \frac{\E_{b}[f(Y)\prod_{k=1}^d Y_k^{a_k -  b_k}]}{\E_{b}[\prod_{k=1}^d Y_k^{a_k -  b_k}]}.
	\end{equation}
\end{thm}
\begin{proof}
	We see that 
	\begin{align*}
		\E_a[f(Y)] & = \frac{\int_{\nabla^{d-1}} f(y)\prod_{k=1}^dy_k^{a_k-1}\, dy}{\int_{\nabla^{d-1}} \prod_{k=1}^dy_k^{a_k-1}\, dy} \\
		& = \frac{\int_{\nabla^{d-1}} f(y)\prod_{k=1}^d y_k^{a_k- b_k}\prod_{k=1}^dy_k^{ b_k-1}\, dy}{\int_{\nabla^{d-1}} \prod_{k=1}^dy_k^{b_k-1}\, dy} \times \frac{\int_{\nabla^{d-1}} \prod_{k=1}^dy_k^{ b_k-1}\, dy}{\int_{\nabla^{d-1}} \prod_{k=1}^d y_k^{a_k- b_k} \prod_{k=1}^dy_k^{b_k-1}\, dy} \\
		& = \frac{\E_{ b}[f(Y)\prod_{k=1}^d y_k^{a_k -  b_k}]}{\E_{b}[\prod_{k=1}^d y_k^{a_k - b_k}]},
	\end{align*} 
	where in the intermediate equality we multiplied and divided by $Q_{b}= \int_{\nabla^{d-1}} \prod_{k=1}^dy_k^{ b_k-1}\, dy.$ 
\end{proof}
As we saw in Section~\ref{sec:bar_a1=0}, the case when the sum of the parameters is zero is particularly tractable. Thus a canonical choice for the vector $b$ in the change of measure formula is $b = a - \bar a_1e_1$, in which case $\bar b_1 = 0$. Under this choice \eqref{eqn:com} becomes
\begin{equation} \label{eqn:com_bar_b1=0}
\E_a[f(Y)] = \frac{\E_{a- \bar a_1e_1}[f(Y)Y_1^{\bar a_1}]}{\E_{a- \bar a_1e_1}[Y_1^{\bar a_1}]}.
\end{equation} 
\section{The case $\bar a_1 = -M$} \label{sec:bar_a1=-M}

\subsection{Moments of the $Y_k$'s} \label{sec:positive_moments}
Remarkably, the identities for the negative moments of $Y_1$ when $\bar a_1 = 0$ can be used to derive \emph{positive} moments, up to order $M$, for a GRD($a$) distribution when $\bar a_1 = -M$. This is the content of the next theorem.
\begin{thm}[Moment formulas for $\bar a_1 = -M$] Suppose that $a \in \R^d$ satisfies Assumption~\ref{assum:bar_a} and that $\bar a_1 = -M$ for some $M \in \N$. Then for any $n \in \N^d_0$ with $\bar n_1 \leq M$ we have that 
	\[\E_a\left[\prod_{k=1}^d Y_k^{n_k}\right] = \frac{ \displaystyle  \sum_{m \in \N_0^d(M-\bar n_1)}  {M - \bar n_1\choose m_1,\dots, m_d } \prod_{k=2}^d \frac{\bar a_k}{\bar a_k + \bar m_k + \bar n_k}} {\displaystyle \sum_{m \in \N_0^d(M)} {M \choose m_1,\dots, m_d }\prod_{k=2}^d \frac{ \bar a_k}{\bar a_k + \bar m_k}}. \]
\end{thm}
\begin{proof}
	This follows directly by taking $f(Y) = \prod_{k=1}^d Y_k^{n_k}$ in \eqref{eqn:com_bar_b1=0} and invoking Theorem~\ref{thm:bar_a1=0_moments} to compute the right hand side of \eqref{eqn:com_bar_b1=0}.
\end{proof}
When $M=1$ this formula takes a particularly simple form
\begin{equation} \label{eqn:moment_m}
	\E_a[Y_k] = C^{-1} \prod_{j=2}^{k} \frac{\bar a_j}{\bar a_j + 1}, \qquad \text{where} \qquad C = 1 + \sum_{k=2}^d \prod_{j=2}^k  \frac{ \bar a_j}{\bar a_j + 1}. 
\end{equation}
In particular this formula is invertible, which allows for explicit first moment matching, which can be used to calibrate the parameters to data.
\begin{cor}[First moment matching]
	Let $y \in \nabla^{d-1}$ satisfying $y_1 > y_2 > \dots > y_d$ be given. Define $a \in \R^d$ via
	\[a_k =  \begin{cases} -1 -  \frac{y_{2}}{y_{1} - y_2}, & k=1, \\
		\frac{y_{k}}{y_{k-1} - y_k} - \frac{y_{k+1}}{y_{k} - y_{k+1}}, & k=2,\dots,d-1, \\
		\frac{y_d}{y_{d-1} - y_d} & k = d.
	\end{cases}\]
	Then $a$ satisfies Assumption~\ref{assum:bar_a}, $\bar a_1 = -1$ and $\E_a[Y_k] = y_k$ for $k=1,\dots,d$.
\end{cor} 
\begin{proof}
	This is readily verified by applying \eqref{eqn:moment_m} to this choice of $a$.
\end{proof}
\subsection{An improved change of measure formula}

In the case that $\bar a_1 = -M$ for some $M \in \N$, the denominator of \eqref{eqn:com_bar_b1=0} is explicitly computable courtesy of Theorem~\ref{thm:bar_a1=0_moments}. By writing $1 = (Y_1+\dots+Y_d)^M$ we can also expand the numerator to obtain that
\begin{equation}  \label{eqn:com_calc}
\begin{split}\E_{a+Me_1}\left[\frac{f(Y)}{Y_1^{M}}\right]&  = \sum_{m \in \N_0^d(M)}{M \choose m_1,\dots, m_d }\E_{a+Me_1}\left[f(Y)\frac{\prod_{k=1}^d Y_k^{m_k}}{Y_1^M}\right] \\
	& = \sum_{m \in \N_0^d(M)} {M \choose m_1,\dots, m_d }\E_{a+Me_1}\left[\frac{\prod_{k=1}^d Y_k^{m_k}}{Y_1^M}\right] \E_{a+m}[f(Y)]\\& =  \sum_{m \in \N_0^d(M)} {M \choose m_1,\dots, m_d } \prod_{k=2}^d \frac{\bar a_k}{\bar a_k + \bar m_k} \E_{a+m}[f(Y)],
\end{split} 
\end{equation} 
where the intermediate equality followed from Theorem~\ref{thm:com} (with $a$ taken to be $a+m$ and $b$ taken to be $a+Me_1$ in the notation of the theorem), while the final equality followed from Theorem~\ref{thm:bar_a1=0_moments}\ref{item:ratio_mn} since $\overline{(a+Me_1)}_1 = 0$. This leads us to the following improved change of measure formula.
\begin{thm}[Change of measure v2] \label{thm:com_v2}
	Let $a \in \R^d$ satisfying Assumption~\ref{assum:bar_a} be given and suppose that $\bar a_1 = -M$ for some $M \in \N$. Then we have that
	 \begin{equation} \label{eqn:com_v2}
	 	\E_a[f(Y)] = \sum_{m \in \N_0^d(M)} w_m\E_{a+m}[f(Y)] \quad \text{where} \quad  	w_m = \frac{{M \choose m_1,\dots, m_d }\prod_{k=2}^d \frac{\bar a_k}{\bar a_k + \bar m_k}}{\sum_{m \in \N^d_0(M)}{M \choose m_1,\dots, m_d }\prod_{k=2}^d \frac{\bar a_k}{\bar a_k + \bar m_k}}
	 \end{equation} 
 for any $\P_a$-integrable function $f:\nabla^{d-1} \to \R$.
\end{thm}
Since the $w_m$'s appearing in \eqref{eqn:com_v2} are positive weights which sum to one, Theorem~\ref{thm:com_v2} establishes that $\P_a$ can be explicitly represented as a mixture of GRD distributions with parameters that sum to zero. This relationship can be leveraged to obtain certain moment formulas for the weights and log gaps, which are explored in the sections below. Additionally, marginal distributions for the weights under the GRD($a$) distribution can be studied with this change of measure identity as well, though we do not pursue this direction in detail here. 

\subsection{The log gaps as a mixture of exponential random variables} \label{sec:log_gaps}
The change of measure formula of Theorem~\ref{thm:com_v2} is particularly insightful when we consider the log gaps $Z_k = \log Y_{k-1} - \log Y_k$ for $k=2,\dots,d$. Indeed, since $Z$ is a function of $Y$, we readily obtain the following corollary to Theorem~\ref{thm:com_v2}.
\begin{cor}[Change of measure for log gaps]  \label{cor:com_gaps} Let $a \in \R^d$ satisfying Assumption~\ref{assum:bar_a} be given and suppose that $\bar a_1 = -M$ for some $M \in \N$. For any function $g:\R^{d-1}_+ \to \R$ such that $g(Z)$ is $\P_a$-integrable we have
	\begin{equation} \label{eqn:com_gaps}\E_{a}[g(Z)] = \sum_{m \in \N^d_0(M)} w_m\E_{a+m}[g(Z)],
	\end{equation} 
	where $w_m$ is defined in \eqref{eqn:com_v2}. In particular the the log gaps $(Z_2,\dots,Z_d)$ under $\P_a$ are a mixture of independent exponential random vectors.
\end{cor}
\begin{proof}
	The formula \eqref{eqn:com_gaps} is a direct consequence of Theorem~\ref{thm:com_v2}, while the claim regarding the mixture of independent exponential distributions follows from Proposition~\ref{cor:bar_a1=0} and the fact that $\bar a_1 + \bar m_1 = 0$ for every $m \in \N^d_0(M)$.
\end{proof}
As an application of Corollary~\ref{cor:com_gaps} we obtain the moment generating function and moments of the log gaps.

\begin{cor}[Log gap moments] \label{thm:gap_moments}
Let $a \in \R^d$ satisfying Assumption~\ref{assum:bar_a} be given and suppose that $\bar a_1 = -M$ for some $M \in \N$. Set $C = \E_{a+Me_1}[1/Y_1^M]$, which is explicitly given by \eqref{eqn:1/y_1_m} since $(\overline{a+Me_1})_1 = 0$. Then 
	\begin{enumerate}
		\item the moment generating function of the log gaps $Z_2,\dots,Z_d$ is given by
		\[\hspace{-0.2cm} \E_a[e^{t_2Z_2 + \dots + t_dZ_d}] = C^{-1}\sum_{m \in \N_0^d(M)} {M\choose m_1,\dots, m_d }\prod_{k=2}^d \frac{ \bar a_k}{\bar a_k - t_k + \bar m_k}; \quad t_k < \bar a_k \quad \text{for }  k=2,\dots,d,\]
		\item for any $n = (n_2,\dots,n_d) \in \N_0^{d-1}$ we have that 
		\[\E_a\left[\prod_{k=2}^d Z_k^{n_k}\right] = C^{-1}\sum_{m \in \N_0^d(M)} {M\choose m_1,\dots, m_d }\prod_{k=2}^d \frac{ \bar a_k n_k!}{(\bar a_k + \bar m_k)^{n_k+1}}. \]
	\end{enumerate}
\end{cor}
\begin{proof}
	This follows directly from Corollary~\ref{cor:com_gaps} and known formulas for exponential random variables.
\end{proof}
\subsection{Generation of random variates} \label{sec:simulate}
We finish Section~\ref{sec:bar_a1=-M} by discussing a way to simulate a random vector $Y$ following a $\P_a$ distribution when $\bar a_1 = -M$. This can be done by first simulating the log gap random vector $Z$ under $\P_a$ using the relationship in Corollary~\ref{cor:com_gaps} and then inverting the maps $Y \mapsto (Z_2,\dots,Z_d) = (\log Y_{1}-\log Y_2,\dots,\log Y_{d-1}- \log Y_d)$. To carry this out we define a random variable $V$ on $\N^d_0(M)$ via $\P(V=m) = w_m$. The simulation steps are then as follows

\begin{minipage}{.7\linewidth} 
\begin{algorithm}[H] \scriptsize
	\caption{Simulating GRD($a$) when $\bar a_1 = -M$}\label{alg:sim_int}
\begin{algorithmic}[1]

	\State $m \gets V$   \hfill ($\triangleright$) sample $V$
		\State Initialize vector $Z = [Z_2,\dots,Z_d]$
	\For{$k=2,\dots,d$}
		\State Simulate one variate from $\mathrm{Exp}(\bar a_k + \bar m_k)$ and store in $Z_k$
	\EndFor
	\State $Y_1 \gets (1 +\sum_{k=2}^d\exp(-\sum_{j=2}^k Z_j))^{-1}$
	\For{$k=2,\dots,d$}
	\State $Y_k \gets Y_{k-1}\exp(-Z_k)$
	\EndFor
\end{algorithmic}
\end{algorithm}
\end{minipage}

\vspace{0.2cm}

\noindent  This ensures that $Y \sim \P_a$. We note that the presentation of the algorithm above is simply pseudocode and the implementation can be made more efficient by vectorizing the operations.
\section{The General Case} \label{sec:general}
In the case that $\bar a_1 \ne -M$ the change of measure formula can still be used to study the GRD distributions. Indeed, by applying Newton's generalized binomial theorem we can obtain a series representation $\E_a[Y_1^{-r}]$ for arbitrary $r \in \R$ in the case $\bar a_1 = 0$.
\begin{prop}[Expected powers of $Y_1$] \label{prop:Y1_series} 
	Let $a \in \R^d$ satisfying Assumption~\ref{assum:bar_a} be given and suppose that $\bar a_1 = 0$. Then for any $r \in \R$ we have
	\begin{equation} \label{eqn:com_series}
		\E_a\left[\frac{1}{Y_1^r}\right] = \sum_{k=0}^\infty {r \choose k} \sum_{j=0}^k {k \choose j} (-1)^{k-j}d^{r-j}\sum_{m \in \N_0^d(j)}{j \choose m_1,\dots,m_d} \prod_{i=2}^d \frac{\bar a_i}{\bar a_i + \bar m_i}.
	\end{equation} 
\end{prop}
\begin{proof}
	We write $1/Y_1 = d(1+ \frac{1-dY_1}{dY_1})$. Note that since $1/d \leq Y_1 \leq 1$ we have that $|\frac{1-dY_1}{dY_1}| < 1$. Hence, applying Newton's binomial theorem and taking expectation yields
	\[\E_a\left[\frac{1}{Y_1^r}\right] = d^r \sum_{k=0}^\infty {r \choose k} \E_a\left[\left(\frac{1}{dY_1}-1\right)^k\right].  \]
	Now applying the standard binomial theorem to the term inside the expectation and using the identity derived in Theorem~\ref{thm:bar_a1=0_moments}\ref{item:1/y_1_m} completes the proof. 
\end{proof}
We now combine this with the change of measure formula to obtain the following theorem.
	\begin{thm}[Change of measure series representation]\label{thm:com_series}
	Let $a \in \R^d$ satisfying Assumption~\ref{assum:bar_a} be given and suppose that $\bar a_1 = -r$ for some $r \in \R$. Then for any $\P_a$-integrable function $f:\nabla^{d-1}\to \R$ we have that 
	\begin{equation} \label{eqn:countable_mixture}
		\E_a[f(Y)] = \sum_{k=0}^\infty \sum_{j=0}^k\sum_{m \in \N_0^d(j)} w_m^{r,j,k}\E_{a+ m +(r-j)e_1}[f(Y)],
	\end{equation} 
	where 
	\[w_m^{r,j,k} = C^{-1}{r \choose k}{k \choose j}(-1)^{k-j}d^{r-j}{j \choose m_1,\dots,m_d}\prod_{i=2}^d \frac{\bar a_i}{\bar a_i + \bar m_i}\] 
	and $C = \E_{a+re_1}[1/Y_1^r]$ is given explicitly by \eqref{eqn:com_series}.
	\end{thm}
\begin{proof}
	From the change of measure identity \eqref{eqn:com_bar_b1=0} we have that
	\[\E_a[f(Y)] = \frac{\E_{a+re_1}[f(Y)Y_1^{-r}]}{\E_{a+re_1}[Y_1^{-r}]}\]
	The denominator has the series representation given by Proposition~\ref{prop:Y1_series}. To handle the numerator we use Newton's binomial theorem to expand out $Y_1^{-r} = d(1+ \frac{1-dY_1}{dY_1})$ as before, multiply both sides by $f(Y)$ and take expectation to obtain
	\begin{equation}  \label{eqn:series_calc}
	\begin{split}
		\E_{a+re_1}[f(Y)Y_1^{-r}]& = d^r\sum_{k=0}^\infty {r \choose k} \E_a\left[f(Y)\left(\frac{1}{dY_1}-1\right)^k\right] \\
		& = \sum_{k=0}^\infty {r \choose k}\sum_{j=0}^k {k \choose j} (-1)^{k-j}d^{r-j}\E_{a+re_1}[f(Y)Y_1^{-j}],
	\end{split}
	\end{equation} 
where we used the standard binomial theorem in the final equality.
Proceeding as in \eqref{eqn:com_calc} we obtain
\begin{align*}
	\E_{a+re_1}\left[\frac{f(Y)}{Y_1^{j}}\right] & = \E_{a+re_1}\left[f(Y)\frac{(Y_1+\dots+Y_d)^j}{Y_1^j}\right] 
	 = \sum_{m \in \N_0^d(j)} {j \choose m_1,\dots,m_d} \E_{a+re_1}\left[f(Y)\frac{\prod_{k=1}^d Y_k^{m_k}}{Y_1^j}\right] \\
	 &= \sum_{m \in \N_0^d(j)} {j \choose m_1,\dots,m_d}\E_{a+re_1}\left[ \frac{\prod_{l=1}^d Y_l^{m_l}}{Y_1^j}\right]\E_{a+m+(r-j)e_1}[f(Y)]  \\
	 & = \sum_{m \in \N_0^d(j)} {j \choose m_1,\dots,m_d} \prod_{i=2}^d \frac{\bar a_i}{\bar a_i + \bar m_i}\E_{a+m+(r-j)e_1}[f(Y)].
\end{align*}
Plugging this into \eqref{eqn:series_calc} completes the proof.
\end{proof}

The upshot of this theorem is that we can represent an arbitrary GRD($a$) distribution as a countable mixture of GRD distributions where the parameter vectors sum to zero. Applying this to the log gap process $Z$ as in Section~\ref{sec:log_gaps} shows, in turn, that the log gaps under an arbitrary GRD($a$) distribution are a countable mixture of independent exponential random variables. This leads to series representation formulas for the log generating function and moments of the log gaps.
\begin{cor}[Log gap moments series representation] Let $a \in \R^d$ satisfying Assumption~\ref{assum:bar_a} be given.
Then 
\begin{enumerate}
	\item the moment generating function of the log gaps $Z_2,\dots,Z_d$ is given by
	\[\E_a[e^{t_2Z_2 + \dots + t_dZ_d}]  = \sum_{k=0}^\infty \sum_{j=0}^k\sum_{m \in \N_0^d(j)} w_m^{- \bar a_1,j,k} \prod_{i=2}^d \frac{\bar a_i}{\bar a_i - t_i + \bar m_i}, \quad t_i < \bar a_i \quad \text{for } i =2,\dots,d,\]
	\item for any $n = (n_2,\dots,n_d) \in \N_0^{d-1}$ we have that 
	\[\E_a\left[\prod_{k=2}^d Z_k^{n_k}\right] = \sum_{k=0}^\infty \sum_{j=0}^k\sum_{m \in \N_0^d(j)} w_m^{- \bar a_1,j,k} \prod_{i=2}^d \frac{\bar a_in_i!}{(\bar a_i + \bar m_i)^{n_i+1}} \]
\end{enumerate}
where $w_m^{-{\bar a_1},j,k}$ is defined in the statement of Theorem~\ref{thm:com_series}.
\end{cor}
Moreover, the representation of $Z$ as a countable mixture of independent exponential random variables suggests an approximate algorithm for generating random GRD($a$) variates for arbitrary parameter $a$ by truncating the series appearing in \eqref{eqn:countable_mixture}. If we keep the first $K+1 \in \N$ terms in the series then by rearranging the terms in the sum we obtain from \eqref{eqn:countable_mixture} that
\[ \E_{a}[f(Y)] \approx \sum_{j=0}^K \sum_{m \in \N_0^d(j)}\tilde w_m^{-\bar a_1,j}(K)\E_{a+m+(-\bar a_1-j)}[f(Y)],\]
where
\[\tilde w_m^{-\bar a_1,j}(K) =  \frac{\sum_{k=0}^K{r \choose k}{k \choose j}(-1)^{k-j}d^{r-j}{j \choose m_1,\dots,m_d}\prod_{i=2}^d \frac{\bar a_i}{\bar a_i + \bar m_i}}{\sum_{j =0}^K \sum_{m \in \N^d_0(j)}\sum_{k=0}^K{r \choose k}{k \choose j}(-1)^{k-j}d^{r-j}{j \choose m_1,\dots,m_d}\prod_{i=2}^d \frac{\bar a_i}{\bar a_i + \bar m_i}}. \]  
Consequently, if we define the random variable $V^K$ on the discrete set $\{m \in \N_0^d: \bar m_1 \leq K\}$ via
\[\P(V^K = m) = w_m^{-\bar a_1, \bar m_1}(K)\] then we obtain an algorithm to approximately sample from the GRD($a$) distribution for arbitrary parameter $a$.
 
\begin{minipage}{.7\linewidth} 
	\begin{algorithm}[H]
		\caption{Simulating GRD($a$) in the general case}\label{alg:sim_gen}
		\begin{algorithmic}[1] \scriptsize
			\Require $K \in \N$
			\State $m \gets V^K$   \hfill ($\triangleright$) sample $V^K$
			\State Initialize vector $Z = [Z_2,\dots,Z_d]$
			\For{$k=2,\dots,d$}
			\State Simulate one variate from $\mathrm{Exp}(\bar a_k + \bar m_k)$ and store in $Z_k$
			\EndFor
			\State $Y_1 \gets (1 +\sum_{k=2}^d\exp(-\sum_{j=2}^k Z_j))^{-1}$
			\For{$k=2,\dots,d$}
			\State $Y_k \gets Y_{k-1}\exp(-Z_k)$
			\EndFor
		\end{algorithmic}
	\end{algorithm}
\end{minipage}

\section{Conclusion}
We introduced the family GRD($a$) of distributions on the ordered simplex $\nabla^{d-1}$. We established change of measure formulas that relate GRD($a$) distributions with different parameters to each other. In the case that $\bar a_1 = -M$ for some $M \in \N$ we exploited the change of measure identity to show that such a distribution is a (finite) mixture of GRD distributions with parameters that sum to zero. This, together with the fact that the log gaps $Z$ are independent exponential random variables when the parameters sum to zero, was used to establish moment formulas, up to order $M$, for the weights as well as for moments of all orders for the log gaps. This led to an algorithm which allows one to \emph{exactly} sample the weights $Y$. In the case $M =1$, the first moment formula is invertible allowing for explicit moment matching which can be used for calibration to data. In the general case when $\bar a_1 \in \R$, we were able to recover many of the same properties, but under series representations rather than finite sums. This led us to an algorithm for approximately sampling the weights $Y$ in this case.

\paragraph{Acknowledgements.} I am grateful to Martin Larsson for helpful discussions.

%
%

\bibliographystyle{plain}
\bibliography{references}

\end{document}